\documentclass{article}

\def\cal{\mathcal}
\usepackage{color}
\usepackage{amsthm,amssymb,amsfonts,mathrsfs}
\usepackage{amsmath}

\newtheorem{theorem}{Theorem}[section]
\newtheorem{proposition}{Proposition}[section]
\newtheorem{lemma}{Lemma}[section]
\newtheorem{corollary}{Corollary}[section]
\theoremstyle{definition}

\newtheorem{definition}{Definition}[section]
\newtheorem{remark}{Remark}

\newtheorem{Example}{Example}

\newtheorem{Assumptions}{Hypothesis}[section]

\def\R{{\mathbb{R}}}
\def\N{{\mathbb{N}}}

\def\itau12{\int_{\tau_1}^{\tau_2}}

\def\ve{\varepsilon}

\def\ds{\displaystyle}

\newcommand{\cY}{{\mathcal Y}}
\newcommand{\cA}{{\mathcal A}}

\title {Boundary controllability for a degenerate wave equation in non divergence form with drift}

\author{{\sc Idriss Boutaayamou}\\
Lab-SIV,
Facult\'e Polydisciplinaires de Ouarzazate,\\
Universit\'e Ibn Zohr,\\
B.P. 638, Ouarzazate 45000, Morocco
\\  email:
dsboutaayamou@gmail.com\\
{\sc Genni Fragnelli}\\
Department  of Ecology and Biology\\ Tuscia University\\ Largo dell'Universit\`a, 01100 Viterbo - Italy\\ email: genni.fragnelli@uniba.it\\
{\sc Dimitri Mugnai}\\
Department  of Ecology and Biology\\ Tuscia University\\ Largo dell'Universit\`a, 01100 Viterbo - Italy\\ email:
dimitri.mugnai@unitus.it}

\date{}
\begin{document}

\maketitle

\begin{abstract}
We consider a degenerate wave equation with drift in presence of a leading operator which is not in divergence form. We provide some conditions for the boundary controllability for the solution of the associated Cauchy problem at a sufficiently large time.
\end{abstract}



\section{Introduction}

It is well know that standard linear theory for transverse waves in a string of length $L$  under tension leads to the classical wave equation
\[
\rho(x)\frac{\partial^2 u}{\partial t^2}(t,x)=\frac{\partial \mathcal T}{\partial x}(t,x)\frac{\partial u}{\partial x}(t,x)+\mathcal  T(t,x)\frac{\partial^2 u(t,x)}{\partial x^2}.
\]
Here $u(t,x)$ denotes the vertical displacement of the string from the $x$ axis at position $x\in (0,L)$ and time $t>0$, $\rho(x)$ is the mass density of the string at position $x$, while $\mathcal T(t,x)$ denotes the tension in the string at position $x$ and time $t$. Divide by $\rho(x)$, assume $\mathcal T$ is independent of $t$, and set  $a(x)=\mathcal T(x)\rho^{-1}(x)$, $b(x)=\mathcal T'(x)\rho^{-1}(x)$. In this way, we obtain
\[
\frac{\partial^2 u}{\partial t^2}(t,x)=a(x)\frac{\partial^2 u(t,x)}{\partial x^2}+b(x)\frac{\partial u}{\partial x}(t,x).
\]
Now, suppose that the density is extremely large at some point, say $x=0$, then the previous equation {\sl degenerates} at $x=0$, in the sense that we can consider $a(0)=0$, while the remainder term is a drift one.

Starting from this model, we consider the  problem
\begin{equation}\label{wave}
\begin{cases}
u_{tt} - a(x)u_{xx} -b(x)u_x=0, & (t,x) \in Q_T,
\\
u(t, 0)= 0, \quad u(t,1)=f(t), & t >0,
\\
u(0,x)= u_0(x), \quad u_t(0,x)=u_1(x),& x \in (0, 1),
\end{cases}
\end{equation}
\\[5pt]
where $Q_T=(0,T) \times(0,1)$, $T>0$. The main feature in this problem is that $a$ degenerates at $x=0$, and the function $f$ acts as a boundary control and it is used to drive the solution to 0 at a sufficiently large time $T>0$ .

This problem has not been treated before, but its study seems natural because of the model above. Moreover, it is related to other ones: the problem without drift and with leading operator in divergence form was considered in \cite{alabau,zg}, but it is well known that problems in non divergence or in divergence form require different settings. Moreover, the drift term $bu_x$ cannot be considered just as a small perturbation of the diffusion term $au_{xx}$ (see \cite[Chapter 7.2]{favya}). These three motivations are the main ones to investigate problem \eqref{wave}.

We will look for a control $f \in L^2(0,T)$, knowing that $a, b\in C^0[0,1]$, $a>0$ on
$(0,1]$ and $a(0)=0$. Concerning $b$, it can possibly degenerate at $0$. Indeed, we will just require that $\ds \frac{b}{a} \in L^1(0,1)$: hence, if $a(x)=x^K$, $K>0$, we can consider $b(x)=x^h$ for any $h>K-1$. Finally, the initial data $u_0$ and $u_1$ belong to suitable weighted spaces.

As for the function $a$, we consider two cases, according to the following
\begin{definition}
A function $a$ is {\it weakly degenerate at $0$}, (WD) for short, if $a\in C^0[0,1]\cap C^1(0,1]$ is such that
$a(0)=0$, $a>0$ on $(0,1]$ and, if 
\begin{equation}\label{stima_a}
\sup_{x \in (0,1]}\frac{x|a'(x)|}{a(x)} = K,
\end{equation}
then $K\in (0,1)$.
\end{definition}

\begin{definition}
A function $a$ is {\it strongly degenerate at $0$}, (SD) for short, if $a\in C^1[0,1]$ is such that
$a(0)=0$, $a>0$ on $(0,1]$ and  in
\eqref{stima_a} we have $K\in [1,2)$.
\end{definition}

We shall always assume that $K<2$, since, as in the parabolic case, observability (and thus null controllability) no longer holds true if $K \ge 2$, see Section \ref{section3} below.

\begin{remark}
Clearly, when $a$ is strongly degenerate we cannot treat the case $b$ strictly positive, since we would have $\ds \frac{1}{a} \not \in L^1(0,1)$ Hence, if $a$ is strongly degenerate at $0$, then $b$ must degenerate at the same point, as well.
\end{remark}

Before going on, let us frame \eqref{wave}. Controllability issues for parabolic problems have been a mainstream topic in
recent years: starting from the
heat equation, related contributions have
been found in more general situations. A common strategy in showing controllability is to prove global Carleman estimates  for
the operator which is the adjoint of the given one.
Such estimates for uniformly parabolic operators without degeneracies or
singularities have been largely developed (see, e.g., \cite{fi}).
Recently, these estimates have been also studied for operators which are not
uniformly parabolic. Indeed,  many problems
coming from Physics and Biology (see \cite{KaZo}), Biology (see \cite{bf},  \cite{bfm}, \cite{epma}, \cite{f2018},\cite{f2020} and \cite{fy}) or Mathematical Finance
(see \cite{HW}) are described by degenerate parabolic equations. In this framework, new Carleman estimates and null controllability properties have been established in \cite{acf}, \cite{cmv2005}, \cite{cmv0}, \cite{fm2013} and \cite{mv} for regular degenerate coefficients, in \cite{bfm} and \cite{fm2016}, \cite{corrigendum} for non smooth degenerate coefficients and in \cite{fs}, \cite{fs1},  \cite{f2016}, \cite{fm2017}, \cite{fm2018}, \cite{fm2020} and in \cite{v} for degenerate and singular coefficients. 

On the contrary, null controllability for degenerate wave equations have received less attention. Let us recall that  null controllability for the one dimensional {\it nondegenerate} wave equation can be attacked in several ways: for instance,  consider
\begin{equation}\label{classica}
\begin{cases}
u_{tt} - u_{xx}=f_\omega(t,x), & (t,x) \in Q_T,
\\
u(t, 0)= 0, \quad u(t,1)=f(t), & t \in (0,T),
\\
u(0,x)= u_0(x)\in H^1_0(0,1), \quad u_t(0,x)=u_1(x)\in L^2(0,1),& x \in (0, 1).
\end{cases}
\end{equation}
Here $u$ is the state, while $f_\omega$ and $f$ are the controls: one may have $f=0$ and $f_\omega$ acting as a control localized in the subset $\omega$ of $[0,1]$, or $f_\omega=0$ and $f$ acting as a boundary control. In any case, one looks for conditions in order to drive the solution to equilibrium at a given time $T$, i.e. given the initial data $(u_0, u_1)$ in a suitable space, we look for a control ($f$ or $f_\omega$) such that
\begin{equation}\label{NC}
u(T,x)=u_t(T,x)=0, \quad \text{ for all } x \in (0,1). 
\end{equation}

Some comments are in order. First, due to the finite speed of propagation of solutions of the wave equation, we cannot expect to have null controllability at any final time $T$ (as in the parabolic case), but we need $T$ sufficiently large: for equation \eqref{classica} it is well known that null controllability holds if $T>2$, see \cite[Chapter 4]{Russell}. Second, the Hilbert Uniqueness Method (HUM) permits to characterize such a control in terms of minimum of a certain functional. In details, following \cite{lions}, in analogy with problem \eqref{wave}, consider \eqref{classica} with  $f_\omega=0$ and its adjoint system
\begin{equation}\label{AD}
\begin{cases}
\varphi_{tt} - \varphi_{xx}=0, & (t,x) \in Q_T,
\\
\varphi(t, 0)= \varphi(t,1)=0, & t \in (0,T),
\\
\varphi(0,x)= \varphi_0(x)\in H^1_0(0,1), \quad \varphi_t(0,x)=\varphi_1(x)\in L^2(0,1)& x \in (0,1).
\end{cases}
\end{equation}
Then, \eqref{NC} is equivalent to the validity of the following  {\it observability inequality}
\begin{equation}\label{sopra}
\|(\varphi_0,\varphi_1)\|^2_{H^1_0(0,1)\times L^2(0,1)}\leq C\int_0^T |\varphi_x(t,1)|^2\,dt,
\end{equation}
where $C=C(T)>0$ is a universal constant. Such an inequality guarantees the coercivity of the functional $J:H^1_0(0,1)\times L^2(0,1)\to \R$ defined as
\[
J(\varphi_0,\varphi_1)=\frac{1}{2}\int_0^T|\varphi_x(t,1)|^2dt+\int_0^1\varphi_1u_0dx-\langle \varphi_0,u_1\rangle_{H^{-1},H^1_0},
\]
where $\varphi$ solves \eqref{AD} with initial data $(\varphi_0,\varphi_1)$. As a consequence, $J$ has a unique minimizer $(\tilde\varphi_0,\tilde\varphi_1)$. Then, the control $f$ which drives the solution of \eqref{classica} to 0 at time $T$ is given by
\[
f(t)=\tilde\varphi_x(t.1),
\]
where $\tilde \varphi$ is the solution of \eqref{AD} with initial data  $(\tilde\varphi_0,\tilde\varphi_1)$. A related approach 
for \eqref{classica} with $f=0$ is showed in \cite{zuazua1}.

In \cite{zuazua0} the authors investigate the problem of maximizing the observability constant, or its asymptotic average
in time, over all possible subsets $\omega$ of $[0, \pi]$ of Lebesgue measure $L\pi$. 
In \cite{o} the author considers the controllability problem of a semilinear wave equation with a control multiplying the  nonlinear term.

To our knowledge, \cite{gu} is  the first paper where a degenerate wave equation is considered: more precisely, the author considers the equation in divergence form
\[
u_{tt}-(x^\alpha u_x)_x=0
\]
for $\alpha\in (0,1)$ and the control acts in the degeneracy point $x=0$. Later on, in \cite{alabau} the authors consider the equation
\[
u_{tt}-(a(x)u_x)_x=0
\]
where $a\sim x^K$. In this  case the authors establish observability
inequalities for weakly  as well as strongly degenerate equations. They
also prove a negative result when the diffusion coefficient degenerates too violently (i.e. when $K \geq 2$) and
the blow-up of the observability time when $K \uparrow 2$. We remark that the observability inequality, and hence null controllability, in \cite{gu} is obtained via spectral analysis, while in \cite{alabau} via suitable energy estimates.
We also mention the recent paper \cite{s}, where the author studies \eqref{classica} with two linearly moving endpoints, establishing observability results in a sharp time and deriving exact boundary controllability results.
\medskip

As far as we know, this is the first paper where the equation is in non divergence form and $b \not \equiv 0$. Clearly the presence of the drift term leads us to use different spaces with respect to the ones in \cite{alabau} and gives rise to some new difficulties. However, thanks to some suitable assumptions on the drift term, one can prove some estimates that are crucial to prove an observability similar to \eqref{sopra}, namely 
\begin{equation}\label{OI}
E_\varphi(0) \le C_T \int_0^T |\varphi_x(t,1)|^2 dt,
\end{equation}
where $C_T$ is a strictly positive universal constant.  Here $E_\varphi$ denotes the energy of $\varphi$, solution of \eqref{AD}, see \eqref{def_energy} below for its precise definition.

\textcolor{black}{We remark that in our situation, the non divergence form of the leading operator forces to use reference Hilbert spaces which are different from those in \cite{alabau}. Indeed, the natural framework must take into account such a function, and so we are forced to introduce weighted spaces with a singular weight ($\sigma$), which gives some complications in the boundary behaviour (see Lemma \ref{lemmalimite}). It must be noticed that such a weight appears independently of the presence of the drift term and in fact it reduces to $a$ when no drift is given.}

\textcolor{black}{As a consequence, the tools used in \cite{alabau} cannot be simply adapted. For instance, the choice of  the multipliers done therein  in some cases doesn't work without any further assumption. For this reason,  in the (SD) case an additional assumption is in order, see Hypothesis \ref{Ass3new}.}

\textcolor{black}{Moreover, the presence of the drift term introduces another complication. Indeed,  without assuming any other special behaviour on the drift term $b$ (such as sign or growth conditions), we cannot expect that necessarily it is an irrelevant lower order term. As a consequence, in order to prove controllability, we will need to assume a "smallness" condition on $b$, which permits to balance the presence of the degenerate coefficient $a$ (see the assumptions in Corollaries \ref{Observability0} and \ref{cor4}).}

\medskip

The paper is organized as follows: in Section \ref{section2} we study well posedness of the original problem with Dirichlet boundary conditions; in Section \ref{section3} we consider the adjoint problem of \eqref{wave} and we prove that for this kind of problem the associated energy is constant through the time. Thanks to this property and to some estimates established in Section \ref{section3}, we  will prove in Section \ref{section5} that \eqref{wave} is null controllable under suitable assumptions.
\medskip

\noindent Notations:\\
 $C$ denotes universal positive constants which are allowed to vary from line to line;\\
$'$ denotes the derivative of a function depending on the real space variable $x$; \\
$\dot\ $ denotes the derivative of a function depending on the real time variable $t$;\\
$y_x^2$ means $(y_x)^2$.

\section{Well posedness for the problem with homogeneous Dirichlet boundary conditions}\label{section2}

We start assuming very low regularity conditions  which will be assumed from now on, but for the functional setting the assumption below is sufficient, while more restrive ones will be needed below.
\begin{Assumptions}\label{basic} Functions $a$ and $b$ are continuous in $[0,1]$ with
$\dfrac{b}{a}\in L^1(0,1)$.
\end{Assumptions}

Consider the degenerate hyperbolic problem
 with Dirichlet boundary conditions
\begin{equation}\label{adjointy}
\begin{cases}
y_{tt} -ay_{xx}-by_x=0, & (t,x) \in (0,+\infty)\times(0,1),\\
y(t,1)=y(t,0)=0, & t \in (0, +\infty),\\
y(0,x) = y^0_T(x), & x \in (0,1),\\
y_t(0,x) = y^1_T(x), & x \in (0,1).
\end{cases}
\end{equation}
We anticipate the fact that the choice of denoting initial data with $T-$dependence is related to the approach for null controllability used in Section \ref{section5}.

In order to study the well-posedness of \eqref{adjointy}, let us
recall the well-known absolutely continuous weight function
\[
\eta(x):=\exp\left \{\int_{\frac{1}{2}}^x\frac{b(s)}{a(s)}ds\right \}
, \quad x\in [0,1],
\]
introduced by Feller in a related context \cite{F1} and used by
several authors, see \cite{cfr}, \cite{favya} and \cite{mp}. Since $\frac{b}{a}\in L^1(0,1)$, we find that $\eta\in C^0[0,1]\cap
C^1(0,1]$ is a strictly positive function. When $b$ degenerates at 0 not slower than $a$, for instance if $a(x)=x^K$ and $b(x)=x^h$,  $K\leq h$, then $\eta$ can be extended to a function of class $C^1[0,1]$. Now set
\[
\sigma(x):=a(x)\eta^{-1}(x),
\]
and observe that if $y$ is sufficiently smooth, e.g. $y \in
W^{2,1}_{\text{loc}}(0,1)$, then
\begin{equation}\label{defsigma}
Ay:=ay_{xx}+by_x= \sigma(\eta y_x)_x.
\end{equation}

As in \cite{cfr}, we introduce
the following Hilbert spaces with related inner products:
\[
 L^2_{\frac{1}{\sigma}}(0,1) :=\left\{ u \in L^2(0,1)\; \big|\; \|u\|_{ \frac{1}{\sigma}}<\infty \right\},
 \;  \langle u,v\rangle_{\frac{1}{\sigma}}:= \int_0^1u v\frac{1}{\sigma}dx;
\]
\[
H^1_{\frac{1}{\sigma}}(0,1) :=L^2_{\frac{1}{\sigma}}(0,1)\cap
H^1_0(0,1),
\;  \langle u,v\rangle_1 :=   \langle u,v\rangle_{\frac{1}{\sigma}} + \int_0^1 u'v'dx;
\]
\[
H^2_{\frac{1}{\sigma}}(0,1) := \Big\{ u \in
H^1_{\frac{1}{\sigma}}(0,1)\; \big|\;Au \in
L^2_{\frac{1}{\sigma}}(0,1)\Big\},
\;   \langle u,v\rangle_2 := \langle u,v\rangle_1+  \langle Au,Av\rangle_{\frac{1}{\sigma}}.
\]

It is clear that the following integration by parts holds.
\begin{lemma}
If $u\in H^2_{{\frac{1}{\sigma}}}(0,1)$ and $v\in H^1_{{\frac{1}{\sigma}}}(0,1)$, then
\[
\int_0^1(\eta u')'v\,dx=-\int_0^1\eta u'v'dx.
\]
\end{lemma}

It turns out that, under reasonable assumptions,  the norm in $H^1_{\frac{1}{\sigma}}(0,1)$  is equivalent to the standard norm
$\int_0^1|u'(x)|^2 dx$ for all $u \in H^1_{\frac{1}{\sigma}}(0,1)$ (see Corollary \ref{equivalenze}).

\begin{Assumptions}\label{Ass0}
Hypothesis $\ref{basic}$ holds. In addition, $a\in C^0[0,1]$ is such that $a(0)=0$, $a>0$ on
$(0,1]$ and there exists $K\in (0,2)$ such that  the function
\[
 x \longmapsto\dfrac{x^K}{a(x)}
\]
is nondecreasing in a right neighborhood of $x=0$.
\end{Assumptions}
Notice that here we don't require any differentiability and that the monotonicity property (which holds globally in $(0,1]$ if $a$ is (WD) or (SD)) is valid only near $0$. Moreover, it is easy to see that the same monotonicity is inherited by the map
\[
x\longmapsto \dfrac{x^\gamma}{a(x)}\]
for any $\gamma\geq K$.

\textcolor{black}{The next result is related to a similar one given in \cite{cfr1} for problems without drift. The proof being similar to that one, we omit it.}
 \begin{proposition}[Hardy-Poincar\'e Inequality]\label{H1a}
 Assume  Hypothesis $\ref{Ass0}$. Then, there exists
 $C>0$ such that
 \begin{equation}\int_0^1 v^2 \frac{1}{\sigma} dx
 \le C \int_0^1(v')^2dx\quad \forall \; v\in H^1_{\frac{1}{\sigma}}(0,1).
 \end{equation}
 \end{proposition}

\begin{remark}
It is clear that the previous inequality is closely related to the classical Hardy inequality
\[
\int_0^1\frac{y^2}{x^2} dx \le 4 \int_0^1y_x^2dx  \mbox{ for any $y\in H^1_0(0,1)$,}
\]
which clearly holds true, since $H^1_{\frac{1}{\sigma}}(0,1)\subset H^1_0(0,1).$
\end{remark}

Thanks to the previous proposition, one has that the
 spaces $H^1_0(0,1)$ and $H^1_{\frac{1}{\sigma}}(0,1)$ algebraically coincide. In particular, we have  the equivalence below.
\begin{corollary}
\label{equivalenze}Assume Hypothesis $\ref{Ass0}$. Then the two norms
\[
\|u\|_1^2:= \int_0^1u^2\frac{1}{\sigma}dx +\int_0^1 (u')^2dx
\quad\mbox{ and }\quad \|u\|_2^2:= \int_0^1 (u')^2dx, 
\]
are equivalent in $H^1_{\frac{1}{\sigma}}(0,1)$.
\end{corollary}

Finally, we introduce the last Hilbert space
\[
\mathcal H_0 := H^1_{{\frac{1}{\sigma}}}(0,1)\times L^2_{{\frac{1}{\sigma}}}(0,1), 
\]
endowed with the inner product
\[
\langle (u, v), (\tilde u, \tilde v) \rangle_{\mathcal H_0}:=  \int_0^1\eta u'\tilde u'dx + \int_0^1 v\tilde v\frac{1}{\sigma}dx
\]
for every $(u, v), (\tilde u, \tilde v)  \in \mathcal H_0$, which induces the norm
\[
\|(u,v)\|_{\mathcal H_0}^2:=  \int_0^1 \eta(u')^2dx + \int_0^1 v^2\frac{1}{\sigma}dx.
\]

\begin{remark}
Notice that the weighted norm in the gradient is equivalent to the usual one, since $\eta$ is bounded above and below by positive constants. \textcolor{black}{However, we introduce such a weight to make considerations straightforward in the steps below, in particular for proving that an involved operator is skew-adjoint.} 
\end{remark}

We are now ready to define the domain $D(A)$ of
the operator $A$ given in \eqref{defsigma} as
\[
D(A) = H^2_{{\frac{1}{\sigma}}}(0,1).
\]
Also, $H^{-1}_{{\frac{1}{\sigma}}}(0,1)$ is the dual space of $H^1_{{\frac{1}{\sigma}}}(0,1)$ with respect to the pivot space $L^2_{{\frac{1}{\sigma}}}(0,1)$. 

In order to study the well posedness of problem \eqref{adjointy}, we introduce the matrix operator $\cA : D(\cA) \subset \mathcal H_0 \rightarrow \mathcal H_0$, given by
\[
\cA:= \begin{pmatrix} 0 & Id\\
A&0 \end{pmatrix}, \quad D(\cA):= H^2_{{\frac{1}{\sigma}}}(0,1) \times H^1_{{\frac{1}{\sigma}}}(0,1).
\]
 In this way,  rewrite \eqref{adjointy} as the Cauchy problem
\begin{equation}\label{CP}
\begin{cases}
\dot \cY (t)= \cA \cY (t), & t \ge 0,\\
\cY(0) = \cY_0,
\end{cases}
\end{equation}
with
\[
\cY(t):= \begin{pmatrix} y\\ y_t \end{pmatrix} \; \text{ and }\; \cY_0:= \begin{pmatrix} y^0_T\\ y^1_T \end{pmatrix}.
\]
Now, it is straightforward to see that for every $\begin{pmatrix} u\\ v \end{pmatrix}\in D(\cA)$ we have
\[
\left<\cA\begin{pmatrix} u\\ v \end{pmatrix},\begin{pmatrix} u\\ v \end{pmatrix}\right>_{\mathcal H_0}=0,
\]
i.e. $\cA$ is a skew-adjoint operator. Moreover, as in the classical case, $I-\cA:D(\cA)\to L^2_{\frac{1}{\sigma}}(0,1)$ is surjective. Hence, by the Lumer-Phillips Theorem, $\cA$ generates a contraction semigroup $(S(t))_{t \ge 0}$ on $\mathcal H_0$.
Hence, if $\cY_0  \in \mathcal H_0$ then $\cY(t)= S(t)\cY_0$ is the mild solution of \eqref{CP}. Also, if $\cY_0 \in D(\mathcal A)$, then the solution is classical and the equation in \eqref{adjointy} holds for all $t \ge0$. Hence, as in \cite{alabau} or in \cite[Proposition 3.15]{daprato}, one has

\begin{theorem}\label{esistenza}
Assume Hypothesis $\ref{Ass0}$.
If $( y^0_T, y^1_T) \in \mathcal H_0$, then there exists a unique mild solution
\[
y \in C^1([0, +\infty); L^2_{\frac{1}{\sigma}}(0,1)) \cap C([0, +\infty); H^1_{\frac{1}{\sigma}} (0,1))
\]
of \eqref{adjointy} which depends continuously on the initial data $(y^0_T, y^1_T) \in  \mathcal H_0.$
Moreover, if  $(y^0_T, y^1_T) \in D(\mathcal A)$, then the solution $y$ is classical, in the sense that
\[
y \in C^2([0, +\infty); L^2_{\frac{1}{\sigma}}(0,1)) \cap C^1([0, +\infty); H^1_{\frac{1}{\sigma}} (0,1)) \cap C([0, +\infty); H^2_{\frac{1}{\sigma}} (0,1))
\]
and the equation of \eqref{adjointy} holds for all $t \ge 0$. 
\end{theorem}

\begin{remark}\label{tminore0}
Due to the reversibility in time of the equation, solutions exist with the same regularity for $t<0$ (indeed, $\cA$ is skew-adjoint and it generates a $C_0$ group of unitary operators on $\mathcal H_0$ by Stone's Theorem). This will be used in Section \ref{section5}.
\end{remark}

Recall that if $a$ is (WD) or (SD), then \eqref{stima_a} implies that the function
\begin{equation}\label{crescente}
x \mapsto \frac{x^\gamma}{a(x)}
\end{equation}
is nondecreasing in $(0,1]$ for all $\gamma \ge K$. In particular, Hypothesis \ref{Ass0} holds. Moreover,
\begin{equation}\label{limite}
\lim_{x\rightarrow 0}  \frac{x^\gamma}{a(x)} =0
\end{equation}
for all $\gamma >K$ and
\begin{equation}\label{binftya1}
\left|\frac{x^\gamma b(x)}{a(x)}\right| \le \frac{1}{a(1)} \|b\|_{L^\infty(0,1)}
\end{equation}
for all $\gamma \ge K$. Let
\begin{equation}\label{M}
 M:= 
\ds \frac{ \|b\|_{L^\infty(0,1)}}{a(1)}.
\end{equation}
These facts will play a crucial role in the following sections.

\section{Energy estimates and boundary observability}\label{section3}
In this section we prove estimates of the energy associated to the solution of the intial value problem both from below and from above. The former will be used in the next section to prove a controllability result, while the latter is used here to prove a boundary observability inequality.

Let $y$ be a mild solution of \eqref{adjointy} and consider its energy 
\begin{equation}\label{def_energy}
E_y(t)=\frac{1}{2}\int_0^1 \left(\frac{1}{\sigma}y_t^2(t,x)  +\eta y_x^2(t,x)\right)dx, \quad \forall \; t \ge 0.
\end{equation}
With the definition above, the classical conservation of the energy still holds true also in the degenerate case with drift:
\begin{theorem}\label{Energiacostante}
 Assume Hypothesis $\ref{basic}$ and let $y$ be a mild solution of \eqref{adjointy}. Then 
\begin{equation}\label{equalityE}
E_y(t) =E_y(0), \quad \forall \; t \ge 0.
\end{equation}
\end{theorem}
\begin{proof}
First of all suppose that $y$ is a classical solution. Then multiply the equation by $\ds\frac{y_t}{\sigma}$ (recall that $y_t \in H^1_{\frac{1}{\sigma}} (0,1)$ by Theorem \ref{esistenza}), integrate over $(0,1)$ and use the boundary conditions to get
\[
\begin{aligned}
0&=\frac{1}{2} \int_0^1\frac{d}{dt} \left(\frac{y^2_t}{\sigma}\right)dx -[\eta y_xy_t]_{x=0}^{x=1} +\int_0^1 \eta y_xy_{tx}dx\\
&=\frac{1}{2}\int_0^1\frac{d}{dt}	\left( \frac{y^2_t}{\sigma}+\eta y^2_x\right)dx =\frac{1}{2}\frac{d}{dt} E_y(t),
\end{aligned}
\]
since $y\in C^2([0, +\infty); L^2_{\frac{1}{\sigma}}(0,1))$ and we can differentiate under the sign integral.

If $y$ is a mild solution, we approximate the initial data with more regular ones, obtaining associated classical solutions (for which \eqref{equalityE} holds), and with the usual estimates we can pass to the limit and obtain the desired result.
\end{proof}
For the rest of the paper the following lemma will be crucial.
\begin{lemma}\label{lemmalimite}
\begin{enumerate}
\item Assume Hypothesis $\ref{basic}$. If $y \in D(A)$ and  $u \in H^1_{\frac{1}{\sigma}}(0,1)$, then
$\ds
\lim_{x \rightarrow 0} u(x) y'(x)=0.
$
\item Assume Hypothesis $\ref{basic}$. If  $u \in D(A)$, then $\displaystyle
\lim_{x\rightarrow 0} x^2 (u'(x))^2=0.$
\item Assume Hypothesis $\ref{Ass0}$. If  $u \in D(A)$ and $K \le 1$, then $\displaystyle\lim_{x\rightarrow 0} x  (u'(x))^2=0$.
\item Assume Hypothesis $\ref{Ass0}$. If  $u \in D(A)$,  $K>1$ and $\ds\frac{xb}{a} \in L^\infty(0,1), $ then $\ds\lim_{x\rightarrow 0} x  (u'(x))^2=0$.
\item Assume Hypothesis $\ref{Ass0}$. If   $u \in H^1_{\frac{1}{\sigma}}(0,1)$, then   $\ds
\lim_{x\rightarrow 0} \frac{x}{a}u^2(x)=0.$
\end{enumerate}
\end{lemma}
\begin{proof}1. Consider the function
$
z(x):=
\eta(x) u(x)  y'(x), \quad x \in (0,1].
$
Observe that
\[
\int_0^1|z|dx \le C \|u\|_{L^2(0,1)}\|y'\|_{L^2(0,1)}.
\]
Moreover
$
z'(x)= (\eta y')'u + \eta y'u'
$
and, by H\"older's inequality, 
\[
\int_0^1|(\eta y')'u| dx= \int_0^1 \sqrt{\sigma}|(\eta y')'| \frac{|u|}{\sqrt{\sigma}} dx\le  \|\sigma(\eta y')'\|_{L^2_{\frac{1}{\sigma}}(0,1)}\|u\|_{L^2_{\frac{1}{\sigma}}(0,1)}
\]
and
\[
\int_0^1\eta|y'u'|dx\le C\|y'\|_{L^2(0,1)}\|u'\|_{L^2(0,1)}.
\]
Thus $z'$ is summable on $[0,1]$.  Hence $z \in W^{1,1}(0,1) \hookrightarrow C[0,1]$ and there exists
$
\lim_{x \rightarrow 0}z(x)=L \in \R.
$
If $L \neq 0$ there would exist a  neighborhood $\cal I$ of $0$ such that
$
\frac{|L|}{2} \le |\eta y'u|,
$
for all $x \in \cal I$;
but, by  H\"older's inequality,
\[
|u(x)|\le \int_0^x |u'(t) |dt \le  \sqrt{x} \|u'\|_{L^2(0,1)}.
\]
Hence
\[
\frac{|L|}{2} \le |\eta y'u| \le \|\eta\|_{\infty} |y'| \sqrt{x} \|u'\|_{L^2(0,1)}
\]
for all $x \in \cal I$. This would imply that
\[
|y'| \ge \frac{|L|}{2 \|u'\|_{L^2(0,1)} \|\eta\|_{\infty} \sqrt{x} }
\]
in contrast to the fact that $ y' \in L^2(0,1)$. Hence $L=0$ and the conclusion follows since $\eta$ is bounded below by a positive constant.

\textcolor{black}{2. It is a straightforward consequence of the previous point. Indeed, it is enough to choose any  function  $u\in H^1_{\frac{1}{\sigma}}(0,1)$ such that $u(x)=x$ in a right neighborhood of 0, and squaring.}

3. 
 Consider the function $z:=x\eta(x) (u'(x))^2$. We have 
\[
z'=\eta(u')^2+x\eta' (u')^2 + 2x\eta u'u''=\eta(u')^2+2xu'(\eta u')'-x\eta'(u')^2.
\]
Proceeding as before one has that $ x\eta'(u')^2$ and $2xu'(\eta u')'$ belong to $L^1(0,1)$ since $K\leq1$. Hence $z\in W^{1,1}(0,1)$ and $\lim_{x\to 0}z(x)=0.$

4.  Proceed as above, observing that
\[
\left|x\eta'(u')^2\right|=\left|x\eta \frac{b}{a}(u')^2\right|\leq \left\| \frac{xb}{a}\right\|_{L^\infty(0,1)}\left\| \eta\right\|_{L^\infty(0,1)}(u')^2
\]
and that
\[
|xu'(\eta u')'|= \left| \frac{x}{\sqrt{a}}\sqrt{\eta}u' \right|\cdot\left| \sqrt{\sigma}(\eta u')'\right|.
\]

5. If $K\leq 1$, the proof is straightforward, since by Hypothesis \ref{Ass0} we get that $\frac{x}{a}$ is bounded. In the general case, set $\ds z:= \frac{x}{a}u^2(x)$. Then $z \in L^1(0,1)$. Indeed
\[
\int_0^1\frac{x}{a} u^2(x)dx \le \frac{1}{\min_{[0,1]}\eta} \int_0^1 \frac{u^2}{\sigma}dx.
\]
Moreover
$
z'= \frac{u^2}{a}+ 2\frac{xuu'}{a} - \frac{a'x}{a^2}u^2;
$
thus, for a suitable $\ve>0$ given by Hypothesis \ref{Ass0},
\[
\begin{aligned}
\int_0^\ve|z'|dx &\le \frac{1}{\min_{[0,1]}\eta}\int_0^1 \frac{u^2}{\sigma}dx + 2 \left(\int_0^\ve \frac{x^2(u')^2}{a}dx\right)^{\frac{1}{2}}\left( \int_0^1 \frac{u^2}{a}dx\right)^{\frac{1}{2}} + K\int_0^1 \frac{u^2}{a}dx\\
&\le \frac{1+K}{\min_{[0,1]}\eta}\int_0^1 \frac{u^2}{\sigma}dx + \frac{2\ve}{\sqrt{a(\ve)}\min_{[0,1]}\eta}\left(\int_0^1 (u')^2dx\right)^{\frac{1}{2}}\left( \int_0^1 \frac{u^2}{\sigma}dx\right)^{\frac{1}{2}}.
\end{aligned}
\]
This is enough to conclude that  $z \in W^{1,1}(0,1)$ and thus there exists  $\lim_{x\to 0}z(x)=L\in\R$. If $L\neq0$, sufficiently close to $x=0$ we would have that
$
\frac{u^2(x)}{a}\geq \frac{|L|}{2x}\not\in L^1(0,1),
$
while $\ds\frac{u^2}{a} \in L^1(0,1)$, since $\ds\frac{u^2}{\sigma} \in L^1(0,1)$.
\end{proof}

\begin{remark}
The assumption $\frac{xb}{a} \in L^\infty(0,1)$ is automatically satisfied if $K \le 1$ thanks to \eqref{binftya1}.
\end{remark}

Now we  prove an inequality for the energy which we will use in the next section to prove the controllability result. For that, we start proving the following result.
\begin{theorem}\label{Stima1}
Assume $a$ (WD) or (SD) and Hypothesis $\ref{basic}$.
If $y$ is a \textcolor{black}{classical }solution of \eqref{adjointy}, then for any $T > 0$ we have
\begin{equation}\label{uguaglianza}
\begin{aligned}
&\frac{1}{2} \eta(1)  \int_0^Ty_x^2(t,1)dt= \int_0^1\left[\frac{x^2y_xy_t}{\sigma}\right]_{t=0}^{t=T} dx - \frac{1}{2} \int_{Q_T} x^2 \eta \frac{b}{a}  y_x^2 dxdt + \int_{Q_T} x\eta y_x^2 dxdt \\
&+\frac{1}{2} \int_{Q_T}\left(2-\frac{x(a'-b)}{a}\right) \frac{1}{\sigma}xy_t^2 dxdt.
\end{aligned}
\end{equation}
\textcolor{black}{As a consequence, if $y$ is a mild solution, then $y_x(\cdot,1)\in L^2(0,T)$ for every $T>0$ and 
\begin{equation}\label{stima1vera}
 \eta(1)  \int_0^Ty_x^2(t,1)dt\le \left(2\left(2+K+ M\right)T +4 \max\left\{\frac{1}{a(1)}, 1\right\}\right)E_y(0),
\end{equation}
where  $M$ is the constant introduced in \eqref{M}.}
\end{theorem}
\begin{proof}
Multiply the equation in \eqref{adjointy} by $\ds\frac{x^2y_x}{\sigma}$ and integrate over $Q_T$. Integrating by parts  we have
\begin{equation}\label{*new}
\begin{aligned}
0&=\int_0^1\left[\frac{x^2y_xy_t}{\sigma}\right]_{t=0}^{t=T} dx -\frac{1}{2} \int_{Q_T} \frac{x^2}{\sigma}(y_t^2)_xdxdt- \int_0^T \left[\eta x^2y_x^2 \right]_{x=0}^{x=1}dt +\int_{Q_T}2x\eta y_x^2 dxdt \\
&+ \int_{Q_T} x^2\eta y_{x}y_{xx}dxdt\\
&= \int_0^1\left[\frac{x^2y_xy_t}{\sigma}\right]_{t=0}^{t=T} dx -\frac{1}{2}\int_0^T \left[\frac{x^2}{\sigma}y_t^2\right]_{x=0}^{x=1} dt +\frac{1}{2}  \int_{Q_T}\left(\frac{x^2}{\sigma}\right)_xy_t^2dxdt \\
&- \int_0^T \left[\eta x^2y_x^2 \right]_{x=0}^{x=1}dt+2\int_{Q_T}x\eta y_x^2 dxdt + \frac{1}{2}\int_{Q_T} x^2\eta (y_{x}^2)_xdxdt\\
& = \int_0^1\left[\frac{x^2y_xy_t}{\sigma}\right]_{t=0}^{t=T} dx -\frac{1}{2}\int_0^T \left[\frac{x^2}{\sigma}y_t^2\right]_{x=0}^{x=1} dt-\frac{1}{2}\int_0^T \left[\eta x^2 y_x^2\right]_{x=0}^{x=1} dt\\
&+ \int_{Q_T} x\eta y_x^2 dxdt -\frac{1}{2} \int_{Q_T} \eta \frac{b}{a}x^2y_x^2dxdt+ \frac{1}{2}\int_{Q_T} \frac{xy_t^2}{\sigma}\left(2- \frac{x(a'-b)}{a}\right)dxdt.
\end{aligned}
\end{equation}

Thanks to the boundary conditions of $y$ and \eqref{limite}, we have
$
\lim_{x \rightarrow 0}\frac{x^2}{\sigma}y_t^2(t,x)=\lim_{x \rightarrow 0} \frac{x^2}{a}\eta y_t^2(t,x)= 0
$
and $\ds\frac{1}{\sigma(1)}y_t^2(t,1) = \frac{\eta(1)}{a(1)}y_t^2(t,1) =0$, hence 
\begin{equation}\label{4new}
\ds\int_0^T \left[\frac{x^2}{\sigma}y_t^2\right]_{x=0}^{x=1} dt =0.
\end{equation}
In addition, thanks to  Lemma \ref{lemmalimite},
$
\lim_{x \rightarrow 0} x^2 \eta y_x^2(t,x)=0,
$
hence
\begin{equation}\label{5new}
\int_0^T \left[x^2\eta y_x^2\right]_{x=0}^{x=1}dt=  \int_0^T \eta(1) y_x^2(t,1)dt.
\end{equation}
Finally, consider the last boundary term. Observing that, by \eqref{crescente},   one has 
\begin{equation}\label{6new}
\begin{aligned}
\left| \int_0^1\frac{x^2y_x(\tau,x) y_t(\tau,x)}{\sigma} dx \right|& \le \frac{1}{2} \int_0^1\frac{x^4}{\sigma}y_x^2dx+ \frac{1}{2}\int_0^1\frac{y_t^2}{\sigma}dx \\
&\le \frac{1}{2a(1)} \int_0^1\eta y_x^2(\tau,x)dx+ \frac{1}{2}\int_0^1\frac{y_t^2(\tau,x)}{\sigma}dx
\end{aligned}
\end{equation}
for all $\tau \in [0,T]$. By Theorem \ref{Energiacostante}, we get
\begin{equation}\label{7new}
\begin{aligned}
\left| \int_0^1\left[\frac{x^2y_x(\tau,x) y_t(\tau,x)}{\sigma} \right]_{\tau=0}^{\tau=T}dx \right|& \le \frac{1}{2} \frac{1}{a(1)} \int_0^1\eta y_x^2(T,x) dx+ \frac{1}{2}\int_0^1\frac{y_t^2(T,x)}{\sigma}dx\\
&+ \frac{1}{2} \frac{1}{a(1)} \int_0^1\eta y_x^2(0,x) dx+ \frac{1}{2}\int_0^1\frac{y_t^2(0,x)}{\sigma}dx\\
&\le2 \max\left\{\frac{1}{a(1)}, 1\right\}E_y(0).
\end{aligned}
\end{equation}
Now, using the fact that $\ds\frac{x^2}{a}$ is nondecreasing, one has
\[
\frac{1}{2}\left| \int_{Q_T} \eta \frac{bx^2}{a} y_x^2dxdt\right|\le \frac{1}{2a(1)}\int_{Q_T} \eta |b| y_x^2 dxdt \le  \frac{M}{2} \int_{Q_T}\eta y_x^2dxdt.
\]
Using the fact that $x|a'| \le Ka$, by \eqref{crescente}, we find
\[
\begin{aligned}
\left| \frac{1}{2}\int_{Q_T} \frac{xy_t^2}{\sigma}\left(2- \frac{x(a'-b)}{a}\right)dxdt\right|&\le  \frac{1}{2}\int_{Q_T} \frac{xy_t^2}{\sigma}\left(2+ K+ \frac{x|b|}{a}\right)dxdt\\
&\leq  \frac{2+K+M}{2}\int_{Q_T} \frac{y_t^2}{\sigma}dxdt.
\end{aligned}
\]
Finally, we clearly have
\[
\int_{Q_T} x \eta y_x^2dxdt \le \int_{Q_T}\eta y_x^2dxdt.
\]
Hence, by using again Theorem \ref{Energiacostante}, inequality \eqref{stima1vera} holds true
for all $T\ge0$.

\textcolor{black}{Now, let $y$ be the mild solution associated to the initial data $(y_0, y_1) \in \mathcal H_0$. Then, consider a sequence $\{(y_0^n, y_1^n)\}_{n \in  \N} \subset D(\mathcal A)$ that converges to $(y_0, y_1)$ and let $y^n$ be the classical solution of \eqref{adjointy} associated to $(y_0^n, y_1^n)$. Clearly $y^n$ satisfies \eqref{stima1vera} by Theorem \ref{esistenza}; then, we can pass to the limit and conclude.}
\end{proof}
In order to get the observability inequality, we need additional assumptions and different techniques, as described in the following two subsections.

\subsection{The weakly degenerate case and the case $K=1$}
In this subsection we assume the following assumption.

\begin{Assumptions}\label{Ass1}
Hypothesis \ref{basic} holds and $a$ is (WD) or (SD) with $K=1$.
\end{Assumptions}
\begin{remark}
Observe that if we consider the prototypes $a(x)=x^K$ and $b(x)=x^h$ with $K \in (0,1]$ and $h \ge 0$, then the condition $\ds\frac{b}{a} \in L^1(0,1)$ is clearly satisfied when $h > K-1$.
\end{remark}
The next preliminary result holds.

\begin{theorem}\label{Stima2}
Assume Hypothesis $\ref{Ass1}$.
If $y$ is a \textcolor{black}{classical} solution of \eqref{adjointy}, then  for any $T > 0$ we have
\begin{equation}\label{uguaglianza_vecchia}
\begin{aligned}
&\frac{1}{2} \eta(1)  \int_0^Ty_x^2(t,1)dt= \int_0^1\left[\frac{xy_xy_t}{\sigma}\right]_{t=0}^{t=T} dx - \frac{1}{2} \int_{Q_T} x \eta \frac{b}{a}  y_x^2 dxdt + \frac{1}{2} \int_{Q_T} \eta y_x^2 dxdt \\
&+\frac{1}{2} \int_{Q_T}\left(1-\frac{x(a'-b)}{a}\right) \frac{1}{\sigma}y_t^2 dxdt.
\end{aligned}
\end{equation}
\end{theorem}
\begin{proof}
We multiply the equation in \eqref{adjointy} by $\ds\frac{xy_x}{\sigma}$ and we integrate over $Q_T$. Recalling \eqref{defsigma} and integrating by parts,  we have
\begin{equation}\label{*}
\begin{aligned}
0&= \int_0^1 \left[\frac{xy_xy_t}{\sigma}\right]_{t=0}^{t=T} dx - \int_{Q_T} \frac{xy_{xt}y_t}{\sigma}dxdt-\int_{Q_T}x(\eta' y_x +\eta y_{xx})y_x dxdt\\
&=\int_0^1\left[\frac{xy_xy_t}{\sigma}\right]_{t=0}^{t=T} dx -\frac{1}{2} \int_{Q_T} \frac{x}{\sigma}(y_t^2)_xdxdt-\int_{Q_T}x\eta' y_x^2 dxdt -\frac{1}{2} \int_{Q_T} x\eta (y_{x}^2)_x dxdt\\
&= \int_0^1\left[\frac{xy_xy_t}{\sigma}\right]_{t=0}^{t=T} dx-\frac{1}{2}\int_0^T \left[\frac{x}{\sigma}y_t^2\right]_{x=0}^{x=1} dt +\frac{1}{2}  \int_{Q_T}\left(\frac{x}{\sigma}\right)_xy_t^2dxdt-\int_{Q_T}x\eta \frac{b}{a} y_x^2 dxdt  \\
&  -\frac{1}{2}\int_0^T \left[x\eta y_x^2\right]_{x=0}^{x=1}dt+\frac{1}{2} \int_{Q_T}(x\eta)_x y_x^2dxdt.
\end{aligned}
\end{equation}
Now, $\left(\frac{x}{\sigma}\right)_x=1-\frac{x(a'-b)}{a}$ and $(x\eta)_x=\eta+x \eta \frac{b}{a}$. Hence, \eqref{*} reads
\begin{equation}\label{1}
\begin{aligned}
\frac{1}{2}\int_0^T &\left[x\eta y_x^2\right]_{x=0}^{x=1}dt = \int_0^1\left[\frac{xy_xy_t}{\sigma}\right]_{t=0}^{t=T} dx  -\frac{1}{2}\int_0^T \left[\frac{x}{\sigma}y_t^2\right]_{x=0}^{x=1} dt -\int_{Q_T}x\eta \frac{b}{a} y_x^2 dxdt\\
&+\frac{1}{2} \int_{Q_T}\left(1-\frac{x(a'-b)}{a}\right) \frac{1}{\sigma}y_t^2 dxdt
 + \frac{1}{2} \int_{Q_T} \eta y_x^2 dxdt + \frac{1}{2} \int_{Q_T} x \eta \frac{b}{a} y_x^2 dxdt\\
&= \int_0^1\left[\frac{xy_xy_t}{\sigma}\right]_{t=0}^{t=T} dx  -\frac{1}{2}\int_0^T \left[\frac{x}{\sigma}y_t^2\right]_{x=0}^{x=1} dt- \frac{1}{2} \int_{Q_T} x \eta \frac{b}{a}  y_x^2 dxdt \\
&+\frac{1}{2} \int_{Q_T}\left(1-\frac{x(a'-b)}{a}\right) \frac{1}{\sigma}y_t^2 dxdt+ \frac{1}{2} \int_{Q_T} \eta y_x^2 dxdt.
\end{aligned}
\end{equation}

Now, thanks to the boundary conditions on $y$ and to Lemma \ref{lemmalimite}.5, we have
\begin{equation}\label{casdeb}
\lim_{x \rightarrow 0}\frac{x}{\sigma}y_t^2(t,x)=\lim_{x \rightarrow 0} \frac{x}{a}\eta y_t^2(t,x)= 0
\end{equation}
and $\ds\frac{1}{\sigma(1)}y_t^2(t,1)  =0$, so that
$
\int_0^T \left[\frac{x}{\sigma}y_t^2\right]_{x=0}^{x=1} dt =0.
$ In addition, by Lemma \ref{lemmalimite}.3,
\[
\int_0^T \left[x\eta y_x^2\right]_{x=0}^{x=1}dt=  \int_0^T \eta(1) y_x^2(t,1)dt.
\]

Hence \eqref{uguaglianza_vecchia} holds if $y$ is a classical solution.
\end{proof}

The observability inequality we get in this case is the following.
\begin{theorem}\label{teoremanuovo}
Assume Hypothesis $\ref{Ass1}$.
If $y$ is a mild solution of \eqref{adjointy}, then 
\begin{equation}\label{stima2vera}
 \eta(1)  \int_0^Ty_x^2(t,1)dt\ge \left(T(2-K-2M)-
 8\max\left \{1, \frac{1}{a(1)}, \frac{K\max _{[0,1]}\eta}{a(1)\min_{[0,1]} \eta}\right\}\right)E_y(0),
\end{equation}
for any $T> 0$.
\end{theorem}

\begin{proof}
As usual, assume that $y$ is a classical solution.
Multiply the equation in \eqref{adjointy} by $\ds\frac{-K y}{2\sigma}$ and integrate on $Q_T$. Recalling \eqref{defsigma} and integrating by parts,  we have
\[
\begin{aligned}
0&= -\frac{K}{2}\int_{Q_T} y_{tt}\frac{y}{\sigma} dxdt + \frac{K}{2} \int_{Q_T} (\eta y_x)_x ydxdt\\
&= -\frac{K}{2}\int_0^1\left[ \frac{yy_t}{\sigma}\right]_{t=0}^{t=T}dx + \frac{K}{2} \int_{Q_T} \frac{y_t^2}{\sigma}dxdt - \frac{K}{2}\int_{Q_T} \eta y_x^2 dxdt + \frac{K}{2}\int_0^T\left[ \eta y_x y\right]_{x=0}^{x=1} dt.
\end{aligned}
\]
Since $\left[ \eta y_x y\right]_{x=0}^{x=1}=0$ by Lemma \ref{lemmalimite}.1, we have
\[
0=-\frac{K}{2}\int_0^1\left[ \frac{yy_t}{\sigma}\right]_{t=0}^{t=T}dx + \frac{K}{2} \int_{Q_T} \frac{y_t^2}{\sigma}dxdt - \frac{K}{2}\int_{Q_T} \eta y_x^2 dxdt.
\]
Summing the previous equality to \eqref{uguaglianza_vecchia} multiplied by 2 and using the degeneracy condition \eqref{stima_a}, we have
\[
\begin{aligned}
\eta(1)& \int_0^T y_x^2(t,1)dt = 2 \int_0^1 \left[ \frac{xy_xy_t}{\sigma}\right]_{t=0}^{t=T} dx- \frac{K}{2} \int_0^1 \left[ \frac{yy_t}{\sigma}\right]_{t=0}^{t=T}dx + \frac{K}{2} \int_{Q_T} \frac{y_t^2}{\sigma}dxdt\\
& -  \frac{K}{2} \int_{Q_T} \eta y_x^2 dxdt + \int_{Q_T} \frac{y_t^2}{\sigma}dxdt + \int_{Q_T} \eta y_x^2 dxdt- \int_{Q_T} x \eta \frac{b}{a}y_x^2dxdt\\& - \int_{Q_T} \frac{x(a' -b)}{a}\frac{y_t^2}{\sigma} dxdt\\
&\ge 2 \int_0^1 \left[ \frac{xy_xy_t}{\sigma}\right]_{t=0}^{t=T} dx- \frac{K}{2} \int_0^1 \left[ \frac{y_t y}{\sigma}\right]_{t=0}^{t=T}dx 
+ \left(1- \frac{K}{2}\right)\int_{Q_T} \eta y_x^2 dxdt\\&+ \left(1+ \frac{K}{2}\right)\int_{Q_T}\frac{y_t^2}{\sigma} dxdt - K \int_{Q_T}\frac{y_t^2}{\sigma} dxdt + \int_{Q_T} \frac{xb}{a} \frac{y_t^2}{\sigma}dxdt - \int_{Q_T} x \eta \frac{b}{a}y_x^2dxdt\\
& \ge 2 \int_0^1 \left[ \frac{xy_xy_t}{\sigma}\right]_{t=0}^{t=T}dx - \frac{K}{2} \int_0^1 \left[ \frac{y_t y}{\sigma}\right]_{t=0}^{t=T}dx \\
&+ \left(1- \frac{K}{2}-M\right)\int_{Q_T} \eta y_x^2 dxdt+ \left(1- \frac{K}{2}-M\right)\int_{Q_T}\frac{y_t^2}{\sigma} dxdt\\
&= 2 \int_0^1 \left[ \frac{xy_xy_t}{\sigma}\right]_{t=0}^{t=T} dx- \frac{K}{2} \int_0^1 \left[ \frac{y_t y}{\sigma}\right]_{t=0}^{t=T} dx+ (2-K-2M)T E_y(0).
\end{aligned}
\]
Now, consider the boundary terms. Clearly,
\[
2\left|\left(\frac{xy_xy_t}{\sigma}\right)(\tau,x)\right| \le \left( \frac{y_t^2}{\sigma}+ \frac{\eta x^2 y_x^2}{a}\right)(\tau,x) \le \left(\frac{y_t^2}{\sigma}+ \frac{\eta  y_x^2}{a(1)}\right)(\tau,x)
\]
for all $\tau \in [0, T]$. Thus
\begin{equation}\label{servepoi}
\begin{aligned}
2 &\left|\int_0^1 \left[ \frac{xy_xy_t}{\sigma}\right]_{t=0}^{t=T}dx\right|\le  \int_0^1 \frac{y_t^2}{\sigma}(T,x)dx +\frac{1}{a(1)} \int_0^1(\eta  y_x^2)(T,x) dx\\
& +  \int_0^1 \frac{y_t^2}{\sigma}(0,x)dx +\frac{1}{a(1)} \int_0^1(\eta  y_x^2)(0,x) dx\le 4 \max\left\{1, \frac{1}{a(1)}\right\} E_y(0).
\end{aligned}
\end{equation}
Moreover,
\[
\left|\left(\frac{yy_t}{\sigma}\right)(\tau,x)\right| \le \frac{1}{2} \left(\frac{y_t^2}{\sigma}+\frac{\eta y^2}{a}\right)(\tau,x)\le \frac{1}{2} \left(\frac{y_t^2}{\sigma}+ \frac{\eta  y^2}{x^2 a(1)}\right)(\tau,x)
\]
for all $\tau \ge 0$.  Thus, by the classical Hardy inequality and Theorem \ref{Energiacostante}, we get
\[
\begin{aligned}
\frac{K}{2} \left|\int_0^1 \left[ \frac{yy_t}{\sigma}\right]_{t=0}^{t=T}dx\right|&\le\frac{K}{4} \int_0^1\frac{y_t^2}{\sigma}(T,x)dx +\frac{K}{4} \frac{\max_{[0,1]}\eta}{a(1)} \int_0^1\left(\frac{y^2}{x^2}\right)(T,x) dx\\
& + \frac{K}{4}  \int_0^1 \frac{y_t^2}{\sigma}(0,x)dx +\frac{K}{4} \frac{\max_{[0,1]}\eta}{a(1)} \int_0^1\left(\frac{y^2}{x^2}\right)(0,x) dx\\
&\le \frac{K}{4}  \int_0^1\frac{y_t^2}{\sigma}(T,x)dx +\frac{K}{4} \frac{4\max_{[0,1]}\eta}{a(1)\min_{[0,1]}\eta } \int_0^1\eta y_x^2(T,x) dx\\
& +  \frac{K}{4} \int_0^1 \frac{y_t^2}{\sigma}(0,x)dx +\frac{K}{4} \frac{4\max_{[0,1]}\eta}{a(1)\min_{[0,1]}\eta } \int_0^1\eta y_x^2(0,x) dx\\
& \le K\max \left\{1, \frac{4 \max_{[0,1]}\eta}{a(1)\min_{[0,1]}\eta }\right\}E_y(0).
\end{aligned}
\]
Hence
\[
\eta(1) \int_0^T y_x^2(t,1)dt  \ge \left(T(2-K-2M) -\textcolor{black}{8} \max\left\{1, \frac{1}{a(1)},\frac{ K\max_{[0,1]}\eta}{a(1)\min_{[0,1]}\eta} \right\}\right)E_y(0).
\]
Thus the conclusion holds true if $y$ is a classical solution.

\textcolor{black}{Now, if $y$ is a mild solution associated to the initial data $(y_0, y_1) \in \mathcal H_0$, then we proceed as at the end of the proof of Theorem \ref{Stima1}.}
\end{proof}

\subsection{The strongly degenerate case with $\boldsymbol{K>1}$}
Now, we consider the case in which $a$ is (SD) with $K>1$ and we assume that Hypothesis \ref{basic} holds. In order to treat this case, it will be enough to assume the following hypothesis.

\begin{Assumptions}\label{Ass3new}
Hypothesis \ref{basic} holds, $a$ is (SD) with $K>1$ and $\dfrac{xb}{a} \in L^\infty(0,1)$.
\end{Assumptions}

\begin{remark}\label{remlimite}
Observe that in this case the condition $\ds\frac{b}{a}\in L^1(0,1)$ implies $b(0)=0$. \textcolor{black}{Indeed, if $b(0)\neq0$, by continuity we would have $\min|b|=\ve>0$ in a right neighborhood of 0, while  
it is well know that in the strongly degenerate case $\ds\frac{1}{a} \not\in L^1(0,1)$ (see \cite{acf})}. For instance, for the prototypes $a(x)=x^K$ and $b(x)= x^h$, then $h > K-1$, so that $b$ is allowed to be either (WD) or (SD). 
\end{remark}

In order to get an estimate  similar to the one proved in Theorem \ref{teoremanuovo}, namely an observability inequality, we introduce the following constant:
\begin{equation}\label{Minfty}
M_\infty:=\left\|\frac{xb}{a}\right\|_{L^\infty(0,1)}.
\end{equation}

By requiring $\frac{xb}{a}\in L^\infty(0,1)$, all the calculations made in the proofs of Theorems \ref{Stima2} and \ref{teoremanuovo} can be reformulated verbatim, substituting $M$ with $M_\infty$. For instance, with the new assumption, in the proof of the analogous of Theorem \ref{Stima2} one can multiply again the equation by $\dfrac{xy_x}{\sigma}$ and all integrations by parts are justified. For instance, since $y$ is a classical solution, we know that $\sigma(\eta y_x)_x\in L^1_{\frac{1}{\sigma}(0,1)}$, so that
\[
\int_{Q_T}(\eta'y_x+\eta y_xx)xy_xdxdt\in \R.
\]
But $\eta'=\eta\frac{b}{a}$, and thus $x\eta\dfrac{b}{a}y_x^2$  is integrable under Hypothesis \ref{Ass3new}. As a consequence, $\eta xy_xxy_x$ is integrable, as well. Moreover, identity \eqref{casdeb} is valid by using Lemma \ref{lemmalimite}.4, the subsequent integration by parts are equally justified, and the proofs still hold true.

In particular, in the (SD) case Theorem \ref{teoremanuovo} re-writes as
\begin{theorem}\label{Stima2new}
Assume Hypothesis  $\ref{Ass3new}$.Then, for any $T >0$ the  mild solution $y$ of \eqref{adjointy} satisfies
\begin{equation}\label{stima2veranew}
 \eta(1)  \int_0^Ty_x^2(t,1)dt\ge \left(T(2-K-2M_\infty)-
 8\max\left \{1, \frac{1}{a(1)}, \frac{K\max _{[0,1]}\eta}{a(1)\min_{[0,1]} \eta}\right\}\right)E_y(0),
\end{equation}
for any $T> 0$.
 \end{theorem}

 \subsection{Boundary observability and failure of boundary observability}

Following \cite{alabau}, we recall the next definition:
\begin{definition}
Problem \eqref{adjointy} is said to be {\it observable in time $T>0$} via the normal derivative at $x=1$ if there exists a constant $C>0$ such that for any $(y_T^0,y_T^1) \in \mathcal H_0$ the mild solution $y$ of \eqref{adjointy} satisfies
\begin{equation}\label{observable}
C E_y(0) \le \int_0^T y_x^2(t,1)dt .
\end{equation}
Moreover, any constant satisfying \eqref{observable} is called {\it observability constant} for \eqref{adjointy} in time $T$. 
\end{definition}
Setting 
\[
C_T := \sup \{C >0: C  \text{ satisfies } \eqref{observable} \},
\]
we have that \eqref{adjointy} is observable if and only if
\[
C_T = \inf_{(y_T^0,y_T^1) \neq (0,0)} \frac{\int_0^T y_x^2(t,1)dt}{E_y(0)} >0.
\]
The inverse of $C_T$, $c_t:= \ds \frac{1}{C_T}$, is called {\it the cost of observability} (or {\it the cost of control}) in time $T$.

Theorem \ref{teoremanuovo} admits the following straightforward corollary.
\begin{corollary}\label{Observability0}
Assume Hypothesis $\ref{Ass1}$  and suppose that  $K<2-2M$ (and so $M<1$).
If
\begin{equation}\label{cor0}
T>\frac{1}{2-K-2M}8 \max\left\{ 1, \frac{1}{a(1)}, \frac{K\max_{[0,1]} \eta}{a(1)\min_{[0,1]} \eta} \right\},
\end{equation}
then \eqref{adjointy} is observable in time $T$. Moreover
\[
\frac{1}{ \eta(1) }\left(T(2-K-2M)-
8 \max\left \{1, \frac{1}{a(1)}, \frac{K\max _{[0,1]}\eta}{a(1)\min_{[0,1]} \eta}\right\}\right)\le C_T.
\]
 \end{corollary}

\begin{remark}\label{Rem9}
Observe that the assumption $K <2-2M$ is clearly satisfied if $\|b\|_{L^\infty(0,1)}< \frac{a(1)}{2}$. Indeed in this case $2-2M >1$ and $K \le 1$.
\end{remark}

The analogous statement of Corollary \ref{Observability0} in the strongly degenerate case becomes
\textcolor{black}{\begin{corollary}\label{cor4}
Assume Hypothesis $\ref{Ass3new}$ and $K<2-2M_\infty$. If
\begin{equation}\label{Obs1}
T>\frac{1}{2-K-2M_\infty}8 \max\left\{ 1, \frac{1}{a(1)}, \frac{K\max_{[0,1]} \eta}{a(1)\min_{[0,1]} \eta} \right\},
\end{equation}
then \eqref{adjointy} is observable in time $T$. Moreover
\[
\frac{1}{ \eta(1) }\left(T(2-K-2M_\infty)-
8 \max\left \{1, \frac{1}{a(1)}, \frac{K\max _{[0,1]}\eta}{a(1)\min_{[0,1]} \eta}\right\}\right)\le C_T.\]
\end{corollary}}
 
As announced in the introduction, we now prove that boundary observability is no longer true when $K \ge 2$ (recall that $K$ is the constant that appears in \eqref{stima_a}). For this, it is enough to discuss two examples where as $a$ we consider the prototype, i.e. $a(x)= x^K$ with $K \ge 2$ and the pure degenerate case $b \equiv 0$. \textcolor{black}{The examples are based on those given in \cite{cfr1}, and for this reason we will only sketch them, referring to \cite{alabau} for the conclusion on the lack of controllability.}

\textcolor{black}{\begin{Example}\label{example1} Let $K\geq2$ and consider the equation
\begin{equation}\label{ex}
u_{tt} -x^Ku_{xx}=0, \qquad (t,x) \in (0,+\infty)\times (0,1).
\end{equation}
Then, if $K=2$, set 
\[
X:= \int_x^1\frac{dy}{y^{K/2}} \quad \mbox{ and} \quad U(t,X):= \frac{1}{x^{K/4}}u(t,x),
\]
while, if $K=2$, set
\[
X:= -\log x, \quad U(t,X):= \frac{1}{\sqrt{x}}u(t,x).
\]
In this way, equation \eqref{ex} becomes
\[
U_{tt} -U_{XX}+ c(X) U=0  \mbox{ or }U_{tt} -U_{XX}+ \ds\frac{1}{4}U=0, \ (t,X) \in  (0,+\infty)\times (0,+\infty),
\]
respectively. Now, proceeding as in $\cite{alabau}$, one proves that \eqref{ex} is not null controllable at time $T$, exploiting the finite speed of propagation and the fact that the equation is settled in a half-line.
\end{Example}}

\section{Null controllability}\label{section5}
In this section we study the problem of null controllability for \eqref{wave}.  More precisely, given $(u_0, u_1) \in L^2_{{\frac{1}{\sigma}}}(0,1)\times H^{-1}_{\frac{1}{\sigma}}(0,1)$, we look for a control $f \in L^2(0,T)$ such that the solution of \eqref{wave} satisfies \eqref{NC}.

First of all, we give the definition of a solution for \eqref{wave} {\sl by transposition}, which permits low regularity on the notion of solution itself: such a definition is formally obtained by re-writing the equation as $u_{tt}-\sigma(\eta u_x)_x$ thanks to \eqref{defsigma}, multiplying by $\frac{v}{\sigma}$ and  integrating by parts. Precisely:
\begin{definition} Let $f \in L^2_{\text{loc}}[0, +\infty)$ and $(u_0, u_1) \in L^2_{{\frac{1}{\sigma}}}(0,1) \times H^{-1}_{{\frac{1}{\sigma}}}(0,1)$. We say that $u$ is a solution by transposition of \eqref{wave} if
\[
u \in C^1([0, +\infty); H^{-1}_{{\frac{1}{\sigma}}}(0,1)) \cap C([0, +\infty); L^2_{{\frac{1}{\sigma}}}(0,1))
\]
and for all $T>0$
\begin{equation}\label{solution}
\begin{aligned}
\langle u_t(T), v^0_T\rangle_{H^{-1}_{{\frac{1}{\sigma}}}(0,1), H^1_{{\frac{1}{\sigma}}}(0,1)}&- \int_0^1\frac{1}{\sigma} u(T) v^1_Tdx = \langle u_1, v(0)\rangle_{H^{-1}_{{\frac{1}{\sigma}}}(0,1), H^1_{{\frac{1}{\sigma}}}(0,1)}\\
& -\int_0^1 \frac{1}{\sigma}u_0 v_t(0,x)dx-\eta(1) \int_0^Tf(t) v_x(t,1)dt
\end{aligned}
\end{equation}
for all $(v^0_T, v^1_T) \in H^1_{{\frac{1}{\sigma}}}(0,1) \times L^2_{{\frac{1}{\sigma}}}(0,1)$, where $v$ solves the backward problem
\begin{equation}\label{adjoint}
\begin{cases}
v_{tt} -av_{xx}-bv_x=0, & (t,x) \in (0, +\infty) \times (0,1),\\
v(t,1)=v(t,0)=0, & t \in (0, +\infty),\\
v(T,x) = v^0_T(x), & x \in (0,1),\\
v_t(T,x) = v^1_T(x), & x \in (0,1).
\end{cases}
\end{equation}
\end{definition}

Setting $y(t,x):= v(T-t,x)$, one has that $y$ satisfies
\eqref{adjointy} with $y^0_T(x)= v^0_T(x)$ and $y^1_T(x)=- v^1_T(x)$. Hence, thanks to Theorem \ref{esistenza} (and Remark \ref{tminore0}), problem \eqref{adjoint} admits a unique solution 
\[
v \in C^1([0, +\infty); L^2_{\frac{1}{\sigma}}(0,1)) \cap C([0, +\infty); H^1_{\frac{1}{\sigma}} (0,1))
\]
which depends continuously on the initial data $V_T:=( v^0_T, v^1_T) \in \mathcal H_0$.

By Theorem \ref{Energiacostante},  the energy is preserved in our setting, as well, so that the method of transposition done in \cite{alabau} continues to hold thanks to \eqref{stima1vera}, and so there exists a unique solution by transposition $u \in  C^1([0, +\infty); H^{-1}_{{\frac{1}{\sigma}}}(0,1)) \cap C([0, +\infty); L^2_{{\frac{1}{\sigma}}}(0,1))$ of \eqref{wave}, namely a solution of \eqref{solution}.

Now, we are ready to pass to null controllability, recalling that by linearity and reversibility of equation \eqref{wave},  null controllability for any initial data $(u_0,u_1)$ is equivalent to exact controllability, see \cite{alabau}. In order to prove that \eqref{wave} is null controllable, let us start with
\begin{Assumptions}\label{Ass4}
Assume 
\begin{itemize}
\item  Hypothesis \ref{Ass1} with $K<2-2M$ and \eqref{cor0}, or
\item \textcolor{black}{Hypothesis \ref{Ass3new} with $K<2-2M_\infty$ and  \eqref{Obs1}.}
\end{itemize}
\end{Assumptions}

Now, consider the bilinear form $\Lambda:  \mathcal H_0\times \mathcal H_0\rightarrow \R$ defined as 
\[
\Lambda (V_T,W_T) :=\eta(1) \int_0^T v_x(t,1)w_x(t,1) dt,
\]
where $v$ and $w$ are the solutions of \eqref{adjoint} associated to the final data $V_T:=( v^0_T, v^1_T)$ and $W_T:=(w^0_T, w^1_T)$, respectively. The following lemma holds. 
\begin{lemma}\label{lambda}
Assume Hypothesis $\ref{Ass4}$. Then, the bilinear form $\Lambda$ is continuous and coercive.
\end{lemma}
\begin{proof}
By Theorem \ref{Energiacostante}, $E_v$ and $E_w$ are constant in time and thanks to \eqref{stima1vera}, one has that $\Lambda$ is continuous. Indeed, by H\"older's inequality, \textcolor{black}{Theorem \ref{Stima1}}, we get
\[
\begin{aligned}
|\Lambda (V_T,W_T)|& \le \eta (1)\int_0^T\left| v_x(t,1)w_x(t,1) \right| dt \\
&\le \left(\eta(1) \int_0^Tv_x^2(t,1)dt\right)^{\frac{1}{2}}\left(\eta(1)\int_0^T  w_x^2(t,1)dt\right)^{\frac{1}{2}}\\
&\leq CE^{\frac{1}{2}}_v(T)E^{\frac{1}{2}}_w(T)\\
&= C \left( \int_0^1 \frac{(v^1_T)^2(x)}{\sigma}dx +\int_0^1\eta v_x^2(T,x) dx\right)^{\frac{1}{2}}\times\\
&\left( \int_0^1 \frac{(w^1_T)^2(x)}{\sigma}dx + \int_0^1\eta w_x^2(T,x) dx\right)^{\frac{1}{2}}\\
&= C \|(v(T), v_t(T))\|_{\mathcal H_0} \|(w(T), w_t(T))\|_{\mathcal H_0} = C \|V_T\|_{\mathcal H_0} \|W_T\|_{\mathcal H_0}
\end{aligned}
\]
for a positive constant $C$ independent of   $(V_T, W_T) \in \mathcal H_0\times \mathcal H_0$.

In a similar way, one can prove that $\Lambda$ is coercive. Indeed, by Theorem \ref{teoremanuovo} or \ref{Stima2new},  for all $V_T \in \mathcal H_0$, one immediately has
\[
\begin{aligned}
&\Lambda (V_T, V_T) =  \int_0^T \eta(1) v_x^2(t,1)dt \ge  C_T E_v(0)=  C_T E_v(T) \ge C \|V_T\|_{\mathcal H_0}^2,
\end{aligned}
\]
for a positive constant $C$.
\end{proof}

Function $\Lambda$ is used to prove the null controllability for the original problem \eqref{wave}. For this, let us start defining $T_0$ as the lower bound found in Corollaries \ref{Observability0} and \ref{cor4}, which changes according to the different assumptions used therein.

\begin{theorem}\label{principale}
Assume Hypothesis $\ref{Ass4}$. Then, for all $T>T_0$ and for every $(u_0,u_1)$ in $L^2_{\frac{1}{\sigma}}(0,1)\times H^{-1}_{\frac{1}{\sigma}}(0,1)$ there exists a control $f \in L^2(0,T)$  such that the solution of \eqref{wave} satisfies 
\[
u(T, x)=u_t(T,x)=0 \quad \text{for all } x \in (0,1).
\]
\end{theorem}
\begin{proof}
Consider the continuous linear  map $\mathcal L: \mathcal H_0 \rightarrow \R$  defined as
\[
\mathcal L (V_T) := -\int_0^1 \frac{u_0 v_t(0,x)}{\sigma}dx+\langle u_1, v(0)\rangle_{H^{-1}_{{\frac{1}{\sigma}}}(0,1), H^1_{{\frac{1}{\sigma}}}(0,1)},
\]
where $v$ is the solution of \eqref{adjoint} associated to the final data $V_T:=( v^0_T, v^1_T) \in \mathcal H_0$. Thanks to Lemma \ref{lambda} and by the Lax-Milgram Theorem, there exists a unique $\bar V_T \in \mathcal H_0$ such that
\begin{equation}\label{LM}
\Lambda (\bar V_T, W_T)= \mathcal L(W_T) \mbox{ for all $W_T \in \mathcal H_0$.}
\end{equation}
Set $f(t):= \bar v_x(t,1)$, $\bar v$ being the solution of \eqref{adjoint} with initial data $\bar V_T$. Then, by \eqref{LM}
\begin{equation}\label{sostituzione}
\begin{aligned}
\eta(1)\int_0^T  f(t) w_x(t,1)dt&=\eta(1) \int_0^T \bar v_x(t,1)w_x(t,1)dt = \Lambda (\bar V_T, W_T)= \mathcal L(W_T) \\
&= \langle u_1, w(0)\rangle_{H^{-1}_{{\frac{1}{\sigma}}}(0,1), H^1_{{\frac{1}{\sigma}}}(0,1)} -\int_0^1 \frac{1}{\sigma}u_0 w_t(0,x)dx
\end{aligned}
\end{equation}
for all $W_T \in \mathcal H_0$.

Finally, denote by $u$ the solution by transposition of \eqref{wave} associated to the function $f$ just introduced above. Then we have
\begin{equation}\label{solutionpoi}
\begin{aligned}
\eta(1) \int_0^Tf(t) w_x(t,1)dt &= -\langle u_t(T), w^0_T\rangle_{H^{-1}_{{\frac{1}{\sigma}}}(0,1), H^1_{{\frac{1}{\sigma}}}(0,1)}+ \int_0^1\frac{1}{\sigma} u(T) w^1_Tdx \\
 &+ \langle u_1, w(0)\rangle_{H^{-1}_{{\frac{1}{\sigma}}}(0,1), H^1_{{\frac{1}{\sigma}}}(0,1)}- \int_0^1 \frac{1}{\sigma}u_0 w_t(0,x)dx.
\end{aligned}
\end{equation}
By \eqref{sostituzione} and \eqref{solutionpoi}, it follows that
\[
\langle u_t(T), w^0_T\rangle_{H^{-1}_{{\frac{1}{\sigma}}}(0,1), H^1_{{\frac{1}{\sigma}}}(0,1)} - \int_0^1\frac{1}{\sigma} u(T) w^1_Tdx=0
\]
for all $(w^0_T, w^1_T) \in \mathcal H_0$. Hence, we have
\[
u(T,x)=u_t(T,x)=0 \quad \text{for all } x \in (0,1).
\]
\end{proof}

\section*{Acknowledgments}
I. B. acknowledges support from {\it ACRI Young Investigator Training Program 2018}. 
G. F. and D. M. are members of  {\it Gruppo Nazionale per l'Analisi Ma\-te\-matica, la Probabilit\`a e le loro Applicazioni (GNAMPA)} of  Istituto Nazionale di Alta Matematica (INdAM). G.F.  is a member of {\it UMI ``Modellistica Socio-Epidemiologica (MSE)''} and is supported by FFABR {\it Fondo per il finanziamento delle attivit\`a base di ricerca} 2017, by  PRIN 2017-2019 {\it Qualitative and quantitative aspects of nonlinear PDEs} and by the
HORIZON$_-$EU$_-$DM737 project 2022 {\it COntrollability of PDEs in the Applied Sciences (COPS)} at Tuscia University. D. M. is supported by INdAM-GNAMPA Project 2022  {\it Elliptic PDE's with mixed diffusion}, by FFABR {\it Fondo per il finanzia\-mento delle attivit\`a base di ricerca} 2017 and by the
HORIZON$_-$EU$_-$DM737 project 2022 {\it COntrollability of PDEs in the Applied Sciences (COPS)} at Tuscia University.

\bibliographystyle{siamplain}
\bibliography{references}

\begin{thebibliography}{99}
\bibitem{alabau} F. Alabau-Boussouira, P. Cannarsa, G. Leugering, {\it Control and stabilization of degenerate wave equations}, SIAM J. Control Optim., \textbf{55} (2017), 2052--2087.
\bibitem{acf}
F. Alabau-Boussouira, P. Cannarsa, G Fragnelli, {\sl Carleman
estimates for degenerate parabolic operators with applications to
null controllability}, J. Evol. Equ. \textbf{6} (2006), 161--204.

\bibitem{daprato} A. Bensoussan, G. Da Prato, M. C. Delfour, S.K. Mitter, {\it Representation and Control of Infinite Dimensional Systems} 2nd ed., Birkh\"auser, Boston 2007.

\bibitem{bf} I. Boutaayamou, G. Fragnelli, {\it A degenerate population system: Carleman estimates and controllability}, Nonlinear Analysis,
\textbf{195} (2020), 111742.

\bibitem{bfm} I. Boutaayamou, G. Fragnelli, L. Maniar, {\it Carleman estimates for parabolic equations with interior degeneracy and
Neumann boundary conditions}, J. Anal. Math., \textbf{135} (2018), 1--35.
\bibitem{b}
H. Brezis, {\sl Functional Analysis, Sobolev Spaces and Partial
Differential Equations}, Springer Science+Business Media, LLC 2011.

\bibitem{cfr1} P. Cannarsa, G. Fragnelli, D. Rocchetti, {\it Controllability results for a class of one-dimensional degenerate
parabolic problems in nondivergence form}, J. Evol. Equ. \textbf{8} (2008), 583--616.
\bibitem{cfr} P. Cannarsa, G. Fragnelli, D. Rocchetti, {\it Null controllability of degenerate parabolic operators with drift}, Netw. Heterog. Media, \textbf{2}  (2007), 693--713.

\bibitem{cmv2005} P. Cannarsa, P. Martinez, J. Vancostenoble, {\it Null controllability of the degenerate
heat equations}, Adv. Differential Equations \textbf{10} (2005), 153--190.

\bibitem{cmv0}
P. Cannarsa, P. Martinez, J. Vancostenoble, {\it Carleman estimates for a class of degenerate parabolic operators},  SIAM J. Control Optim. {\bf 47} (2008), 1--19.

\bibitem{epma}
C.L. Epstein, R. Mazzeo, {\it Degenerate Diffusion Operators Arising in Po\-pu\-lation Biology}, Ann. of Math. Stud., 2013.

\bibitem{favya}
A. Favini, A. Yagi,  {\it Degenerate Differential Equations in Banach Spaces}, Pure and Applied Mathematics: A Series of
   Monographs and Textbooks, \textbf{215}, M.Dekker, New York, 1998.


\bibitem{F1}   
  W. Feller,  {\it The parabolic differential equations and the
   associated semigroups of transformations}, Ann. Math., \textbf{55} (1952),
   468--519.
\bibitem{fs}
M. Fotouhi, L. Salimi, {\it Controllability results for a class of one dimensional
degenerate/singular parabolic equations}, Commun. Pure Appl. Anal.
\textbf{12} (2013), 1415--1430.

\bibitem{fs1}
M. Fotouhi, L. Salimi, {\it Null controllability of
degenerate/singular parabolic equations}, J. Dyn. Control Syst. {\bf 18}  (2012), 573--602.

\bibitem{f2016}
G. Fragnelli, {\it Interior degenerate/singular parabolic equations in nondivergence form: well-posedness and Carleman estimates},
J. Differential Equations \textbf{260} (2016), 1314--1371.
\bibitem{f2018}
G. Fragnelli, {\it Carleman estimates and null controllability for a degenerate population model}, J. Math. Pures Appl. \textbf{115}
(2018), 74--126.
\bibitem{f2020}
G. Fragnelli, {\it Null controllability for a degenerate population model in divergence form via Carleman estimates}, Adv.
Nonlinear Anal. \textbf{9}
(2020), 1102--1129.
\bibitem{fm2013} G. Fragnelli, D.  Mugnai, {\it Carleman estimates and observability
inequalities for parabolic equations with interior degeneracy},
Adv. Nonlinear Anal. \textbf{2}, (2013), 339--378.
\bibitem{fm2016}
G. Fragnelli, D. Mugnai, {\it  Carleman estimates, observability inequalities and null controllability for interior degenerate
non smooth parabolic equations}, Mem. Amer. Math. Soc., \textbf{242} (2016), no. 1146, v+84 pp.
\bibitem{fm2017} G. Fragnelli, D. Mugnai, {\it Carleman estimates for singular parabolic equations with interior degeneracy and non smooth
coefficients}, Adv. Nonlinear Anal. \textbf{6} (2017), 61--84.
\bibitem{fm2018}
G. Fragnelli, D. Mugnai, {\it Controllability of strongly degenerate parabolic problems with strongly singular potentials},
Electron. J. Qual. Theory Differ. Equ. \textbf{50} (2018), 1--11.

\bibitem{corrigendum}
G. Fragnelli, D. Mugnai, {\it Corrigendum and Improvements to "Carleman estimates, observability inequalities and null controllability for interior degenerate non smooth parabolic equations"}, Mem. Amer. Math. Soc., \textbf{272} (2021).

\bibitem{fm2020}
G. Fragnelli, D. Mugnai, {\it Singular parabolic equations with interior degeneracy and non smooth coefficients: the Neumann
case}, Discrete Contin. Dyn. Syst.-S, DOI: 10.3934/dcdss.2020084.
\bibitem{fy}
G. Fragnelli, M. Yamamoto, {\it Carleman estimates and controllability for a degenerate structured population model}, Appl. Math. Optim. ,  \textbf{84} (2021), 207--225
\bibitem{fi}
A. V. Fursikov, O. Yu. Imanuvilov, {\it Controllability of evolution
equations, Lecture Notes Series, Research Institute of Mathematics},
Global Analysis Research Center, Seoul National University \textbf{34}, 1996.

\bibitem{gu}
M. Gueye, {\it Exact boundary controllability of 1-d parabolic and hyperbolic degenerate equations}, SIAM J. Control Optim.
{\bf 52}(2014), 2037--2054.

\bibitem{HW}
P. Hagan, D. Woodward, {\it Equivalent Black volatilities}, Appl. Math. Finance {\bf 6} (1999), 147-159.
\bibitem{KaZo}
N.I. Karachalios, N.B. Zographopoulos, {\it On the dynamics of a degenerate parabolic equation: global bifurcation of stationary states and convergence},
Calc. Var. Partial Differential Equations
{\bf 25} (2006), 361--393.

\bibitem{lions}
J.L. Lions, {\it Contr\^olabilit\'e exacte, perturbations et stabilisation de syst\`emes distribu\-es}. Tome 1: Contr\^olabilit\-e exacte. Rech. Math. Appl. \textbf{8}. Paris etc.: Masson. x, 538 p. (1988).

\bibitem{mv}
P. Martinez, J. Vancostenoble, {\it Carleman estimates for one-dimensional degenerate heat equations}, J. Evol. Eq. {\bf 6} (2006), 325--362.

   \bibitem{mp}   (1743484)
   G. Metafune and D. Pallara,
   {\it Trace formulas for some singular differential operators
   and applications}, Math. Nachr., \textbf{211} (2000), 127--157.
\bibitem{o} M. Ouzahra, {\it Controllability of the semilinear wave equation governed by a multiplicative control}, Evol. Equ. Control Theory \textbf{8} (2019), 669--686. 
\bibitem{zuazua0}  Y. Privat, E. Tr\'elat, E. Zuazua, {\it Optimal observation of the one-dimensional wave equation}, J. Fourier Anal. Appl. \textbf{19} (2013), 514--544
\bibitem{zuazua1} Y. Privat, E. Tr\'elat, E. Zuazua, {\it Optimal location of controllers for the one-dimensional wave equation}, Ann. Inst. H. Poincar\'e Anal. Non Lin\'eaire \textbf{30} (2013), 1097-–1126. 

\bibitem{Russell}
D.L. Russell, {\it Controllability and stabilizability theory for linear partial differential equations:
recent progress and open questions}, SIAM Rev. {\bf20} (1978), 639--739.

\bibitem{s} A. Sengouga, {\it Exact boundary observability and controllabiity of the wave equation in an interval with two moving endpoints}, Evol. Equ. Control Theory \textbf{9} (2020),1--25.
\bibitem{v}
J. Vancostenoble, {\it Improved Hardy-Poincar\'e inequalities and sharp Carleman
estimates for degene\-rate/singular parabolic problems}, Discrete Contin.
Dyn. Syst. Ser. S \textbf{4} (2011), 761--790.

\bibitem{zg}
M. Zhang, H. Gao, {\it Null Controllability of Some Degenerate Wave Equations}, J. Syst. Sci. Complex \textbf{29} (2016), 1--15
   \end{thebibliography}

\end{document}